\documentclass[11pt, leqno]{amsart}
\usepackage{amsmath, amssymb, amsthm, enumerate, bbm}
\usepackage{graphicx}
\usepackage{booktabs}
\usepackage[font=small]{caption} 
\usepackage[flushleft]{threeparttable}
\usepackage{mathtools, breqn}
\usepackage{url}

\def\arXiv#1{arXiv:\href{http://arXiv.org/abs/#1}{#1}}

\def\MR#1{MR\href{http://www.ams.org/mathscinet-getitem?mr=#1}{#1}}

\usepackage[
heightrounded,
top=1in,
bottom=1.3in,
inner=1.1in,
outer=1.1in 
]{geometry} %

\usepackage[colorlinks=true, urlcolor=black, linkcolor=blue, citecolor=magenta,  hypertexnames=false]{hyperref}

\def\C{\mathbb C}

\def\dim{\operatorname{dim}}

\def\ext{\operatorname{ext}}

\def\F{\mathbb F}
\def\FP{\mathbb {FP}}

\def\N{\mathbb N}

\def\R{\mathbb R}

\def\S{\mathbb S}
\def\sp{\operatorname{span}}
\def\supp{\operatorname{supp}}
\def\theta{\vartheta}

\newtheorem{theorem}{Theorem}[section]
\newtheorem{lemma}[theorem]{Lemma}

\newtheorem{proposition}[theorem]{Proposition}
\newtheorem{corollary}[theorem]{Corollary}
\newtheorem{definition}[theorem]{Definition}
\theoremstyle{definition}

\newtheorem{conjecture}[theorem]{Conjecture}

\numberwithin{equation}{section}

\title{Energy on spheres and discreteness of minimizing measures}
\date{\today}
\author{Dmitriy Bilyk}
\author{Alexey Glazyrin}
\author{Ryan Matzke}
\author{Josiah Park}
\author{Oleksandr Vlasiuk}
\address{School of Mathematics, University of Minnesota, Minneapolis, MN 55455} 
\email{dbilyk@math.umn.edu}
\address{School of Mathematical \& Statistical Sciences, The University of Texas Rio Grande Valley, Brownsville, TX 78520}
\email{alexey.glazyrin@utrgv.edu}
\address{School of Mathematics, University of Minnesota, Minneapolis, MN 55455} 
\email{matzk053@umn.edu}
\address{School of Mathematics, Georgia Institute of Technology, Atlanta, GA 30332}
\email{j.park@gatech.edu}
\address{Department of Mathematics, Florida State University, Tallahassee, FL 32306}
\email{ovlasiuk@fsu.edu}
\thanks{}
\subjclass[2010]{Primary 52A40, 31E05; Secondary 58C35, 90C26}

\keywords{Potential energy minimization, spherical codes, spherical designs, attractive-repulsive}

\begin{document} 

\maketitle



\section{Introduction}\label{sec:Intro}

Energy minimization on the sphere arises naturally in numerous contexts in mathematical physics, discrete and metric geometry, coding theory, signal processing, and  other fields of mathematics.  Many problems can be reformulated in terms of  minimization of the {\em{discrete energy}}
\begin{equation}\label{eq:poten}
   E_f (\mathcal C) = \frac{1}{|\mathcal C|^2}  \sum\limits_{x, y\in\mathcal{C}} f( \langle x, y \rangle)
\end{equation} 
over all $N$-point configurations $\mathcal C \subset \S^{d-1}$, or of the continuous {\em{energy integral}} 
\begin{equation}\label{eq:muener}
    I_{f}(\mu)=\int_{\S^{d-1}}\int_{\S^{d-1}} f (\langle x,y \rangle )  d\mu(x)d\mu(y)
\end{equation} 
over $\mu \in \mathcal P (\S^{d-1})$, the set of all Borel probability measures. In this work, we mostly concentrate on the energy integrals.  We assume that the measurable function $f: [-1,1]\rightarrow \R$ is bounded below, hence the integral \eqref{eq:muener} is well defined, although it may be infinite for some measures. 

Loosely speaking, minimizing the discrete $N$-point energy may be interpreted as finding the equilibrium position of $N$ ``particles'' on the sphere, which interact  according to the potential $f$, which depends on the distance between $x$ and $y$, while minimizing  the energy integral corresponds to finding the optimal distribution of unit charge on $\S^{d-1}$ under the same interaction. Minimization of energy integrals dictates the limiting behavior of the discrete problem as the number of points  $N$ goes to infinity. Observe that the interaction depends only on the distance between $x$ and $y$, hence the energy \eqref{eq:muener} is invariant under orthogonal transformations.

The definitions of the discrete  \eqref{eq:poten} and continuous \eqref{eq:muener} energies are compatible in the sense that 
\begin{equation}\label{eq:trans}
E_f (\mathcal C) =  I_f \big(\mu_{\mathcal C}   \big), \,\,\,\textup{ where }\,\,\, \mu_{\mathcal C} = \frac{1}{|\mathcal C|}  \sum_{x\in \mathcal C } \delta_x,
\end{equation}
and we shall often abuse the terminology by saying that $\mathcal C$ (instead of $\mu_{\mathcal C}$) minimizes $I_f$.

In some  models, energy minimization leads to a clustering effect, in the sense that the resulting optimal measures tend to be discrete or at least supported on lower dimensional submanifolds. This phenomenon has been repeatedly observed for energies on $\mathbb R^d$ with {\em attractive-repulsive} potentials, which naturally appear in models in computational chemistry, mathematical biology, and social sciences \cite{balague2013nonlocal,carrillo2017geometry,carrillo2003kinetic,kangkim,kolokolnikov2011stability,mogilner2003mutual,von2012predicting,wu2015nonlocal}. 

In many instances, it is observed that in the Euclidean setting the above energies are minimized by measures supported on a sphere of some radius. Our results have some implications in this direction, but we concentrate primarily  on {\em{attractive-repulsive}} potentials on the sphere, i.e. functions $f( \langle x, y \rangle)$ which are increasing near $1$, but decreasing near $-1$; in other words,   two particles  $x$ and $y$ experience repulsion when $x$ and $y$ are close, but attract when they are far apart. In some  examples, potentials of the energy are also   {\em{symmetric}} and {\em{orthogonalizing}}, i.e. they satisfy $ f(t) = f(|t|)$, and $\min \{ f(t):\, t\in [-1,1] \} = f(0)$, which  results in two particles achieving equilibrium when they are in an orthogonal position.

One of the most interesting energies of this type is the {\em{$p$-frame energy}} corresponding to $f (t) = |t|^p$, where $p>0$, 
\begin{equation}\label{eq:pframe}
    I_{f}(\mu)=\int_{\S^{d-1}}\int_{\S^{d-1}} |\langle x,y \rangle|^{p} d\mu(x)d\mu(y).
\end{equation} 
The behavior of minimizing measures of this energy exhibits peculiar phase transitions at  even integer values of $p$. Whenever $p\in 2 \N$, the $p$-frame energy is minimized by the normalized surface measure $\sigma$ \cite{sidel1974new,Ehler2012}, among other measures.  However, for $p\not\in 2\N$,  all the minimizers appear to be discrete \cite{BGMPV}.

For $p=2$, this energy and  its discrete counterpart, often referred to simply as the {\em{frame potential}},  have been studied in \cite{sidel1974new} and later again in \cite{benedetto2003finite}. In the latter paper, which coined the name for this energy, it was proved that the minimizers of the discrete energies  with $N\ge d$ points  are precisely  unit norm {\em{tight frames}}. 
A tight frame  is a set of vectors $\{\varphi_{i}\}_{i=1}^N\subset \R^{d}$ such that a Parseval type identity, $$\sum\limits_{i=1}^N |\langle x,\varphi_{i} \rangle|^2=A \|x\|^2,$$ holds for all $x\in\R^d$ and for some constant $A>0$. In other words, tight frames   act as overcomplete orthonormal bases and thus play an important role in several areas of applied mathematics. 
 It is easy to see that tight frames, and more generally, isotropic measures on the sphere also minimize the continuous frame energy over all probability measures. 
 
 In the case $p=4$, the $p$-frame energy is closely connected to the maximal equiangular tight frames, which in the complex case are  known as {\em symmetric informationally complete positive operator-valued measures} (SIC-POVMs). 
These are unit norm tight frames $\{ \varphi_i\}_{i=1}^N $ with the property that $|\langle \varphi_i, \varphi_j \rangle |^2 = \frac{1}{d+2}$ or $\frac{1}{d+1}$ for $i\neq j$, in the real and complex case respectively.  In $\C^d$, Zauner's conjecture \cite{zau99} states that SIC-POVMs exist in all dimensions $d\ge 2$, which is supported by extensive numerical evidence \cite{SG,RBSC}. In the real case, the existence of analogous objects is also mysterious: they may exist only in dimensions $d=(2m-1)^2-2$ \cite{LS}, but do not exist for $d=47$ \cite{makhnev}. When these objects do exist, they minimize the $4$-frame energy (with the complex unit sphere replacing $\S^{d-1}$ in the case of SIC-POVMs).

 More generally, when $p = 2k \in 2 \mathbb N$, the function $f(t) = |t|^p = t^{2k}$ is a polynomial, hence any {\em{spherical $2k$-design}} yields the same value of the $p$-frame energy as $\sigma$, and thus is also a minimizer.  More precisely,  discrete equal-weight minimizers are exactly  {\em{projective $k$-designs}}.
 A spherical $t$-design is a set $\{ x_i\}_{i=1}^N \subset \S^{d-1}$   such that $$ \frac{1}{N} \sum_{i=1}^N P(x_i)  = \int_{\S^{d-1}} P(x) d\sigma  (x) $$ for any polynomial $P$ on $\R^d$ of degree up to $t$; see e.g. \cite{delsarteSpherical1977}, while a  projective $k$-design is a configuration such that the above identity holds  for all polynomials of degree up to $2k$, which contain only even-degree terms. To summarize, the $p$-frame energy has  a multitude of minimizers, both continuous and discrete,  when $p$ is an even integer.  
 
 When $p\not\in 2\N$, the situation is much less studied. In our recent paper \cite{BGMPV}, we have shown that, when certain highly symmetrical configurations exist, they minimize the $p$-frame energy on a range of values of $p$ between two consecutive even integers. These configurations, known as {\em{tight designs}},  are designs of high order with few distinct pairwise distances, or equivalently, designs of smallest possible cardinality \cite{delsarteSpherical1977}.  
 Theorem 1.1 in \cite{BGMPV} states that a tight spherical $(2k+1)$-design, whenever it exists,  minimizes the $p$-frame energy for $p\in [2k-2,2k]$, and, moreover,  {\em{every}} minimizer for $p\in (2k-2,2k)$  has to be a tight design (in particular, it has to be discrete). In addition, we have accumulated  a great deal   of numerical evidence that  suggests discreteness of minimizers, leading us to the following conjecture:

 \begin{conjecture}\label{conj1} Let  $p>0$ and $p\not\in 2\N$. Then every minimizer of the $p$-frame energy \eqref{eq:pframe} is a finite discrete measure on $\S^{d-1}$. 
 \end{conjecture}
 
 There are other  conjectures  in the literature also asserting the discreteness of measures minimizing certain energies on the sphere. We mention a couple of examples.
 
Let $f( t ) = \arccos |t|$, i.e. $f( \langle x, y \rangle)$ represents the non-obtuse angle between the lines generated by the vectors $x$ and $y$. A conjecture of Fejes T\'oth \cite{FT} states that the $N$-point  energy \eqref{eq:poten} (the sum of acute angles) is {\em{maximized}}  by the periodically repeated elements of the orthonormal basis, and  the  continuous version of the conjecture speculates that $I_f$ is {\em{maximized}}  by the discrete measure uniformly distributed over the elements of the orthonormal basis (see \cite{BM} for more details and recent results).

Another similar conjecture stems from mathematical physics and relativistic quantum field theory \cite{finsterSupport2013,FS2}. It concerns the {\em{causal variational principle}}, which, in the spherical case, concerns minimizing  the energy  on $\S^2$ with the kernel  
\begin{equation}\label{eq:causal}
 f( \langle x, y \rangle) = \max \{ 0,\, 2\tau^2 ( 1+ \langle x, y \rangle ) \big( 2- \tau^2 (1- \langle x, y \rangle ) \big) \big\}, 
 \end{equation}
 with a real parameter $\tau >0$.  It is conjectured in \cite{finsterSupport2013} that for any $\tau \ge 1$ {\em{there exists}} a discrete minimizer, and for $\tau > \sqrt{2}$ {\em{all minimizers}} are discrete. In \cite{BGMPV} it is demonstrated that for two values of $\tau$, minimizers are the cross-polytope  and the icosahedron, respectively.

In the present paper we prove a series of results which establish discreteness of minimizers or smallness of their support (or at least the existence of such minimizers) for various classes of energies on $\S^{d-1}$.  
In particular, in Theorem \ref{thm:discrete} we prove  a quantitative version of the following statement:
 \begin{theorem}\label{t:discrete1}
 Assume that  $f\in C[-1,1]$ has only  finitely many positive coefficients in its orthogonal expansion with respect to Gegenbauer polynomials $C_n^\lambda$ (with $\lambda=  \frac{d-2}{2}$). Then there exists a discrete minimizer of the energy $I_f$ on $\S^{d-1}$.   
 \end{theorem}
\noindent The cardinality of the support of this discrete minimizer is bounded by the dimension of the space of spherical harmonics, corresponding to the positive coefficients of $f$. 
The proof relies on the analysis of the structure of extreme points of the set of moment-constrained measures.  Section \ref{sec:discrete} contains a self-contained exposition of these arguments.

While the discreteness of the minimizers  claimed in Conjecture \ref{conj1} remains out of reach, we establish that the support of the measures minimizing  the $p-$frame energy with  $p\not\in 2\N$ must be small: 
\begin{theorem}\label{t:pframe}
Assume that $p>0$ and $p\not\in 2\N$, and set $f(t) = |t|^p$. Let $\mu \in \mathcal P (\S^{d-1}) $ be a minimizer of the $p$-frame energy $I_f$ \eqref{eq:pframe}. Then the support of $\mu$ has empty interior, i.e. $$ \big( \operatorname{supp}  \mu \big)^\circ = \emptyset.$$
\end{theorem}

\noindent  Section \ref{sec:Empty interior, p-frame} is devoted to the proof of this theorem. In order to compare this theorem  to some known results on $\R^d$, we  point out that discreteness of minimizers for attractive-repulsive potentials on $\R^d$  has been proved  in \cite{carrillo2017geometry} under the assumption that $f$ is {\em mildly repulsive}, i.e. that the potential,  as a function of $r= |x-y|$, behaves as $-r^\alpha$ for small $r$, with $\alpha >2$. Since on the sphere  $\displaystyle{| \langle x,y \rangle |^p \approx 1 - \frac{p}{2} r^2 }$,  the $p$-frame potential corresponds to the endpoint case $\alpha =2$ and thus is quite  delicate: indeed, we know for some values of $p$ there exist non-discrete minimizers. In the recent paper \cite{LMc} it was shown that for  some specific attractive-repulsive potentials with $\alpha \ge 2$, the corresponding energies are uniquely minimized by discrete measures on regular simplices. The complete understanding of the endpoint case $\alpha =2 $ remains an interesting open problem.
 
 In Section \ref{sec:analytic}, we also prove that an analog of Theorem \ref{t:pframe} holds for energies  with kernels  $f:[-1,1]\rightarrow \R$, which are real-analytic, but  not positive definite on $\S^{d-1}$ up to an additive constant (see Definition \ref{def:pd} and Proposition \ref{p:pd}). Theorem \ref{Analytic_supp} states that  for such kernels, minimizing measures have support with empty interior. Moreover, on the circle $\S^1$, they are discrete.  This generates  a certain dichotomy: for an analytic functions $f$, either the energy $I_f$ is minimized by the uniform surface measure $\sigma$, or {\em{all minimizers}} have support with empty interior.

 This result, as well as Theorem \ref{t:discrete1},   obviously applies to polynomials. Thus, when a polynomial $f$ is not positive definite (up to an additive  constant),  the support of {\em{every}} minimizer has empty interior, while for {\em{every}} polynomial $f$ there exists a discrete minimizer of $I_f$ (see Corollary \ref{cor:poly}). For positive definite polynomials $f$, discrete minimizers are just {\em{weighted spherical designs}}, but for arbitrary polynomials, existence of discrete minimizers is new. Section \ref{sec:poly} presents a discussion of energies with polynomial kernels.

 Finally, in Section \ref{sec:locglob} we present an interesting observation that for positive definite kernels $f$, any local minimizer of the energy $I_f$ is necessarily a global minimizer. This applies, in particular, to the $p$-frame energy with even integer values of $p$ and to many other interesting energies.

\section{Background}\label{sec:Background}
\subsection{Spherical harmonics and Gegenbauer polynomials}

For a parameter $\lambda >0$,  consider the weight $\displaystyle{\nu(t) = (1-t^2)^{\lambda -\frac12}}$ on the interval $[-1,1]$, where from now on $\lambda = \frac{d-2}{2}$. The weight $\displaystyle{\nu(t)}$ is related to integration on the sphere $ \S^{d-1}$ in the following way: for a unit vector $p\in \S^{d-1}$,
\begin{equation}
\int_{\S^{d-1}}  f ( \langle x, p \rangle ) d\sigma (x) = I_f ( \sigma) = \frac{\omega_{d-2}}{\omega_{d-1}} \int_{-1}^1 f(t) \big(1-t^2\big)^{\frac{d-3}{2}} \, dt,
\end{equation}
where, as before, $\sigma $ is the normalized surface measure on the sphere $\S^{d-1}$ and $\omega_{d-1} = \frac{2\pi^{d/2}}{\Gamma (d/2)}$ is the $(d-1)$-dimensional Hausdorff surface measure of $\S^{d-1}$. 

Gegenbauer polynomials $C_n^\lambda$, $n\ge 0$,  form a  sequence of   orthogonal polynomials with respect to the weight $\nu(t)$ on the interval $[-1,1]$. Every function $f \in L^1 \big([-1,1], \nu(t)dt \big)$ has a Gegenbauer (ultraspherical) expansion
\begin{equation}\label{eq:expansion}
f(t) \sim \sum_{n=0}^\infty \widehat{f}_n  \frac{n+\lambda}{\lambda} C_n^\lambda (t).
\end{equation}
For $ f \in L^2 \big([-1,1], \nu(t)dt \big)$ this expansion converges to $f$ in the $L^2$ sense. In the case of $\S^1$, when $\lambda = 0$, the relevant polynomials are the Chebyshev polynomials of the first kind
\begin{equation}
T_n (t) = \cos \big( n \arccos t \big) = \frac12 \lim_{\lambda \rightarrow 0}  \frac{n+\lambda}{\lambda} C_n^\lambda (t),
\end{equation}
and for $\S^2$, the polynomials are appropriately scaled Legendre polynomials \cite{szego}.

Let $\mathcal H_n^d$ denote the space of spherical harmonics of order $n$, the functions which are restrictions to $\S^{d-1}$ of homogeneous harmonic polynomials of degree $n$  on $\R^d$. These spaces are mutually orthogonal  for different values of $n$ and satisfy $$ L^2 \big( \S^{d-1}, d\sigma \big)  = \bigoplus_{n=0}^\infty \mathcal H_n^d. $$
Let $\{ Y_{n,j} \}$  be any orthonormal basis in $\mathcal H_n^d$. The Gegenbauer polynomials are related to the spherical harmonics by the following  {\em{addition formula}} 
\begin{equation}\label{eq:add}
    \sum_{j=1}^{a_n^d} Y_{n,j}(x) Y_{n,j}(y) = \frac{n+\lambda}{\lambda} C_n^\lambda(\langle x,  y\rangle)\   \ \textup{ for all }\,\, x,y\in \S^{d-1},
\end{equation}
where  $$ a_n^d  = \text{dim}\ \mathcal{H}_n^d =\frac{n+\lambda}{\lambda} C_n^\lambda(1) = 
 \frac{2n+d-2}{n+d-2} {{n+d-2} \choose {d-2}} .$$

For more detailed information on spherical harmonics, Gegenbauer polynomials, and harmonic analysis on the sphere, we refer the reader to \cite{daiApproximation2013,mullerSpherical2007}.

\subsection{Positive definite functions on $\S^{d-1}$.}  Positive definite functions play an important role in energy minimization.
\begin{definition}\label{def:pd}
A function  $f \in C [-1,1]$ is positive definite on subset $K$ of sphere $\S^{d-1}$, if for every collection of points $\{x_i \}_{i=1}^N \subset K$, the matrix $ \big[ f \big( \langle z_i, z_j  \rangle \big) \big]_{i,j=1}^N$ is positive semidefinite, i.e. for any  sequence $\{c_i\}_{i=1}^N\subset \C$, $f$ satisfies the inequality $$ \sum\limits_{i,j=1}^N c_i\overline{c}_j f \big( \langle x_i, x_j  \rangle \big)\geq 0.$$
\end{definition}
When $K=\S^{d-1}$ is the entire sphere, positive definite functions admit several equivalent characterizations, which connect this property to Gegenbauer polynomials and energy. 

\begin{proposition}\label{p:pd}
Let $f \in C[-1,1]$. The following statements are equivalent:
\begin{enumerate}[(i)]
\item The function $f$ is positive definite on $\S^{d-1}$. \vspace{1 mm}
\item For any signed Borel measure $\nu$ on $\S^{d-1}$, $ I_f (\nu) \ge 0.$ \vspace{1 mm}
\item\label{part:schoen} The coefficients in the ultraspherical expansion \eqref{eq:expansion} of $f$ with respect to Gegenbauer polynomials $C_n^\lambda $ are non-negative: $\widehat{f}_n \ge 0 \,\, \textup{ for all } \,\, n\ge 0 .$ \vspace{1 mm}
\item\label{part:energy} The minimum of the energy $I_f$ over Borel probability measures on $\S^{d-1}$ is achieved by $\sigma$ and is  non-negative, i.e. $$\min_{\mu \in \mathcal P (\mathbb S^{d-1} )} I_f (\mu ) = I_f (\sigma) \ge 0. $$
\end{enumerate}
\end{proposition}
Part \eqref{part:schoen} of this proposition is a classical result due to Schoenberg \cite{schoenbergPositive1941}. Part \eqref{part:energy} states that positive definite functions (up to additive constants) are precisely those potentials  for which  energy minimization imposes uniform distribution. 

Moreover, positive definiteness implies uniform convergence of the Gegenbauer expansion \eqref{eq:expansion}; see e.g. \cite{gangolli}.
\begin{lemma}\label{l:absconv}
Assume that $f\in C[-1,1]$ is positive definite on $\mathbb S^{d-1}$.  
Then the Gegenbauer expansion \eqref{eq:expansion} converges to $f$ absolutely and uniformly on $[-1,1]$. 
\end{lemma}
One of the simplest ways to prove this statement is using Mercer's theorem from spectral theory on the representation of symmetric positive definite functions \cite{mercers}. In turn, Lemma \ref{l:absconv} together with the addition formula \eqref{eq:add} easily imply part \eqref{part:energy} of Proposition \ref{p:pd}. 

Positive definiteness also plays a role  when the energy is not minimized by the uniform measure $\sigma$. In this case, we  have the following implication \cite{Bjorck1956,finsterSupport2013}. 

\begin{lemma}\label{lem:positive_semidef}
Let  $f \in C([-1,1])$. Assume  that  $\mu$ is a minimizer of $I_f$ over $\mathcal P( \S^{d-1})$ and  $I_f (\mu ) \ge 0$. 
Then the function $f$ must be positive definite on $\operatorname{supp} (\mu)$.
\end{lemma}

\noindent Observe that, together with part  \eqref{part:energy} of Proposition \ref{p:pd},  this immediately implies the following: 
\begin{corollary}\label{cor:dich}
{{Either  $\sigma$ is  a minimizer of $I_f$ (i.e. $f$ is  positive definite on $\S^{d-1}$, up to an additive constant), or every  minimizer of $I_f$ is supported on a proper subset of the sphere $\mathbb S^{d-1}$.}}
\end{corollary} 

Much of this paper is dedicated to obtaining various refinements of this principle for various classes of kernels  $f$. Lemma \ref{lem:positive_semidef} also suggests an approach to proving that a certain set cannot be contained in the support of a minimizer: one may attempt to prove that $f$ is not positive definite on that set. This idea, albeit not in a straightforward fashion, is exploited  in the proof of Theorem \ref{t:pframe} in the next section; see the proof of Proposition \ref{prop:I}. 

\subsection{Gegenbauer expansions and other minimizers}
In some situations, Gegenbauer coefficients can give some information about the minimizers, even when $\sigma$ does not minimize the energy.  Below we mention several relevant results of this type. While we do not use them in this paper, we chose to include them because they are similar in spirit to the results of the paper: they provide certain conditions, under which {\em{there exist}} discrete minimizers or {\em{all}} minimizers are discrete. These results can be found in \cite{bilyk2018geodesic}.

\begin{itemize}
\item If $\widehat f_n \le 0$ for all $n\ge 1$, then a Dirac delta mass $\mu = \delta_z$, for any $z\in \S^{d-1}$, is a minimizer of $I_f$. If $f$ has a strict  absolute minimum at $t=1$ (in particular, if $\widehat f_n < 0$ for all $n\ge 1$), then every minimizer is a Dirac mass. Observe that this case resonates with Theorem \ref{thm:discrete}.
 \item If $(-1)^{n+1} \widehat f_n \ge 0$ for all $n\ge 1$, then  the measure $\mu = \frac12 \big(\delta_z + \delta_{-z}  \big)$ is a minimizer of $I_f$. Moreover, all minimizers are of this form, if  the strict inequality $(-1)^{n+1} \widehat f_n > 0$ holds.  
\item If $\widehat f_{2n} = 0$  and $\widehat f_{2n-1} \ge 0$ for all $n\ge 1$, then every centrally symmetric measure minimizes $I_f$. In particular, there exist discrete minimizers. 
\end{itemize}

We note that for the Euclidean setting, and certain attractive-repulsive potentials, there are classifications of potentials for which two-point measures appear as minimizers, see \cite{kangkim}.

\section{Existence of discrete minimizers}\label{sec:discrete}

\subsection{Extreme points for sets of moment-constrained measures}
In the present section we exhibit a large class of potentials $f$ for which there exist discrete minimizers of the energies $I_f$. The methods that we employ are closely related to {\em{moment problems}}. 

Let $\Omega$ be a compact metric space and let $\mathcal B_+ (\Omega)$ denote the set of positive Borel measures on $\Omega$.  Given continuous functions $f_0,...,f_n$ on $\Omega$ and non-negative constants $c_i$, we consider the set 
 \begin{equation}\label{eq:defk}
     K = \bigg\{ \mu \in \mathcal B_+(\Omega): \int_\Omega f_i d\mu = c_i, \, i=0,1,\dots,n \bigg\},
 \end{equation}
 which consists of  Borel measures whose moments with respect to $f_i \in C(\Omega)$ are fixed. We always set $f_0 \equiv 1$ and $c_0 =1$, so that $\mu \in K$   guarantees  that $\mu $ is a probability measure, i.e. $\mu (\Omega) =1$.

 It is easy to see that $K$ is convex, bounded, and weak-$*$ closed, and therefore is weak-$*$ compact. By the Krein--Milman theorem, $K$ is the weak-$*$ closure of $\operatorname{ext} (K)$ --- the set of extreme points of $K$. The  results  presented below describe  the structure of  $\operatorname{ext} (K)$, in particular, the discreteness of its elements. To make this section self-contained, we    include their proofs.

 We start with a theorem which gives a necessary condition for $\mu$ to be an extreme point of $K$.
 
 \begin{theorem}[Douglas, \cite{douglasExtr1964}]\label{thm:doug}
 Assume that $\mu \in \ext (K)$. Then 
 \begin{equation}
 L^1 (d\mu )  = \sp \{ f_0=1, f_1, \dots, f_n \}. 
 \end{equation}
 \end{theorem}

 \begin{proof}
 Assume that $g\in L^\infty (d\mu )$ satisfies $$ \int_\Omega f_i g  d\mu = 0, \qquad i=0,1,\dots,n. $$
 Multiplying $g$ by a constant, we may assume that $\| g \|_{L^\infty (d\mu )} < 1$. Then the measures $\mu_\pm$,  defined by  $d\mu_{\pm} = (1\pm g)  d\mu$,  belong to $K$, since $\mu_\pm \in \mathcal B_+ (\Omega)$ and  $$ \int _\Omega  f_i d\mu_\pm = \int_\Omega f_i (1 \pm g) d\mu =  \int_\Omega f_i  d\mu = c_i. $$
 At the same time, $\mu = \frac12 (\mu_-+\mu_+)$. Since $\mu \in \ext K$, this implies that $\mu_\pm = \mu$ and hence $g=0$ $\mu$-a.e. Therefore, the functions $f_i$ span $L^1 (d\mu)$. 
 \end{proof}

We now state and prove a result, which demonstrates the discreteness of the elements of $\ext (K)$. This result has a number of precursors and extensions, see \cite{Richter,rogo1,rogo,rose,winkler,zuhov}.

\begin{theorem}[Karr, \cite{karrExtr1983}]\label{thm:karr}
Let $\mu \in K$. Then the following statements are equivalent:
\begin{enumerate}[(i)]
\item\label{G1} $\mu \in \ext (K)$.
\item\label{G2}  The cardinality of $\supp \mu$ is at most $n+1$. Moreover,  if we denote $\supp \mu = \{x_1, \dots, x_k \}$, then the vectors $v_j =  \big(  1, f_1 (x_j), \dots, f_n (x_j) \big)$, $j=1,2,\dots, k$, are linearly independent. 
\end{enumerate}
\end{theorem} 

\begin{proof}
\noindent (\ref{G1})$\Rightarrow$(\ref{G2}).  Assume that there exist points $\{ x_1, \dots, x_{n+2} \} \subset \supp \mu$. Then one can find a vector $y \in \mathbb R^{n+2}$, which is not in the span of the vectors $\big(f_i (x_1),  f_i (x_2), \dots, f_i (x_{n+2}) \big)$, $i=0,1,\dots, n$, since the latter subspace is at most $n+1$ dimensional. Appealing to Urysohn's lemma, one can construct a  function $g \in C(\Omega) \subset L^1 (d\mu)$ such that $g (x_i ) = y_i$  for $i=1,2,\dots, n+2$.  But then $g\not\in \sp \{ f_i\}$, which contradicts Theorem \ref{thm:doug}, i.e. $\big| \supp \mu  \big| \le n+1$. 

Now that it is known that $\mu  = \sum_{i=1}^{k} t_i \delta_{x_i}$ with $k\le n+1$, $t_i > 0$, $\sum t_i =1$, consider the linear system 
\begin{equation}\label{eq:sys}
\begin{pmatrix} 1 & \dots & 1 \\
 f_1 (x_1) & \dots & f_1 (x_k) \\
 \vdots & \ddots & \vdots \\
f_n(x_1)  & \dots & f_n (x_k)  \end{pmatrix}
\begin{pmatrix} \alpha_1 \\ \alpha_2 \\ \vdots \\ \alpha_k    \end{pmatrix}  = 
\begin{pmatrix}  1 \\ c_1 \\ \vdots \\ c_n    \end{pmatrix}.
\end{equation} 
This system has a unique solution $\alpha_i = t_i$, since if the solution is not unique, then there is a whole affine subspace of solutions and one could perturb the values of $t_i$ in opposite directions, i.e.\ find two solutions of the form $\{ t_i \pm \tau_i \}$, and construct two measures $\mu_\pm =   \sum_{i=1}^{k} (t_i \pm \tau_i)  \delta_{x_i}$ so that $\mu_\pm \ge 0 $ and  $\int f_i d\mu_\pm = \int f_i d\mu$, 
i.e. $\mu_\pm \in K $, and  $\mu = \frac12 (\mu_+ + \mu_-) $, which contradicts the fact that $\mu \in \ext (K)$. This proves the linear independence of the rows of the matrix above. \\

\noindent (\ref{G2})$\Rightarrow$(\ref{G1}). Assume that (\ref{G2}) holds. Then the system \eqref{eq:sys} has a unique solution, i.e. $\mu$ is uniquely determined by the the condition $\supp \mu \subset \{x_1, \dots, x_k \}$. If $\mu = \frac12 (\mu_1 + \mu_2)$, then $\supp \mu \subset \supp \mu_1 \cup \supp \mu_2$, and thus $\supp \mu_j \subset \{x_1,\dots, x_k\}$ for $j=1,2$. Therefore $\mu_1=\mu_2 = \mu$, i.e. $\mu \in \ext (K)$.  
\end{proof}

We remark that convex geometry plays heavily into similar characterizations of solutions to infinite dimensional optimization problems in the recent papers \cite{boyer,chandra,unser}.

\subsection{Applications of Karr's theorem: existence of discrete minimizers.}
We now apply the results on moment-constrained  measures to prove that for a function $f$ with only finitely many positive terms in its Gegenbauer expansion, there exist discrete minimizers of $I_f$. 

Let $\widehat f_n$ denote the coefficients in the Gegenbauer expansion  \eqref{eq:expansion} of the  function $f\in C[-1,1]$.  
Consider the sets $N_+ (f) = \{ n\ge 0: \widehat f_n > 0\}$ and $N_- (f) = \{ n\ge 0:  \widehat f_n < 0\}$. We shall assume that 
\begin{equation}\label{eq:finpos}
\big| N_+ (f) \big| <\infty,
\end{equation}
 i.e.\ there are only finitely many terms of \eqref{eq:expansion} with $\widehat f_n >0$.  In this case, the function  $$\displaystyle{\sum_{n\in N_+(f) }  \widehat f_n \,\frac{n+\lambda}{\lambda} C_n^\lambda (t) -  f(t) }$$
is continuous and positive definite. According to Lemma \ref{l:absconv}, this implies that the Gegenbauer expansion \eqref{eq:expansion} of $f$ converges uniformly and absolutely.  

Recall that $\mathcal H_n^d$ denotes the space of spherical harmonics of degree $n$ on $\S^{d-1}$. We are now ready to state the main theorem of the section. 

\begin{theorem}\label{thm:discrete}
Assume that the Gegenbauer expansion \eqref{eq:expansion} of the function $f \in C[-1,1]$ satisfies   $$ \big| N_+ (f)  \big|= \big| \{ n\ge 0:  \widehat f_n > 0\} \big| <\infty, $$ i.e.\ the Gegenbauer expansion has only finitely many positive terms. Then there exists a discrete measure $\mu^* \in \mathcal P (\S^{d-1})$ such that
\begin{equation}\label{eq:suppbound}
\big| \supp \mu^*  \big|  \le \sum_{n\in N_+(f) \cup \{0\} }  \dim \mathcal H_n^d,
\end{equation}
and $\mu^*$ minimizes  the energy $I_f (\mu) $ over $\mathcal P (\S^{d-1})$, i.e.
\begin{equation}
I_f (\mu^*) = \inf_{\mu \in \mathcal P (\S^{d-1})} I_f (\mu).
\end{equation}
\end{theorem}

\begin{proof} 
Let $\nu \in \mathcal P (\S^{d-1})$ be any minimizer of $I_f$ and set $$ \displaystyle{M = \inf_{\mu \in \mathcal P (\S^{d-1})}  I_f (\mu ) = I_f (\nu)}.$$  %
We shall use the addition formula for spherical harmonics \eqref{eq:add},
as well as the absolute convergence of \eqref{eq:expansion}, to re-write for any measure $\mu \in \mathcal P (\S^{d-1})$, $I_f (\mu )$ as,
\begin{align*}
\sum_{n=0}^\infty   \widehat f_n  \int_{\S^{d-1}}   \int_{\S^{d-1}}     \frac{n+\lambda}{\lambda}\ C_n^\lambda  (\langle x, y \rangle )  d\mu(x) d\mu (y) = \sum_{n=0}^\infty   \widehat f_n   \Bigg[ \sum_{j=1}^ {\dim \mathcal H_n^d} \bigg( \int_{\S^{d-1}}  Y_{n,j} (x) d\mu (x) \bigg)^2 \Bigg]\\
= \sum_{n\in N_+ (f) }   \widehat f_n   \Bigg[ \sum_{j=1}^ {\dim \mathcal H_n^d} \bigg( \int_{\S^{d-1}}  Y_{n,j} (x) d\mu (x) \bigg)^2 \Bigg] - \sum_{n\in N_- (f) }  \big(- \widehat f_n\big)   \Bigg[ \sum_{j=1}^ {\dim \mathcal H_n^d} \bigg( \int_{\S^{d-1}}  Y_{n,j} (x) d\mu (x) \bigg)^2 \Bigg],\\
\end{align*}
the last of which we define as the difference of functionals $\mathcal F(\mu ) - \mathcal G(\mu)$. It is easy to see that $\mathcal G$ is convex with respect to $\mu$ since it is a positive linear combination of squares of linear functionals of $\mu$. Let us set $$K = \Bigg\{ \mu \in \mathcal B_+ (\S^{d-1}) :  \int_{\S^{d-1}}  Y_{n,j} d\mu (x) = \int_{\S^{d-1} } Y_{n,j} d\nu (x),\,  n\in N_+ (f),\, j=1,2,\dots, \dim \mathcal H_n^d \Bigg\} , $$ so that $\nu \in K$ and $\mathcal F (\mu) = \mathcal F (\nu) $ for $\mu \in K$. 
Without loss of generality, we shall assume that $0\in N_+ (f)$. This guarantees that $\mu \in K$ is a probability measure (similarly to setting $c_0=1$ and $f_0 \equiv 1$ earlier). Since $N_+ (f) <\infty$, the set $K$ has finitely many moment constraints and Theorem \ref{thm:karr} is applicable. In fact, the number of constraints is exactly the right-hand side of \eqref{eq:suppbound}. 

Given that  $\mathcal G$ is convex in $\mu$ and $K$ is a convex weak-$*$ compact subset of $\mathcal B_+ (\S^{d-1})$, we conclude that $\mathcal G(\mu)$ achieves its maximum on $K$ at a point of $\ext (K)$. Hence there exists a measure $\mu^* \in \ext (K)$ such that $\displaystyle{ \mathcal G (\mu^*) = \sup_{\mu \in K} \mathcal G (\mu)}$. We then find that 
\begin{align*}
M & = I_f (\nu) = \mathcal F(\nu) - \mathcal G(\nu) = \mathcal F(\mu^*) - \mathcal G(\nu) \ge \mathcal  F(\mu^*) - \mathcal G(\mu^*) = I_f (\mu^*) \ge M, 
\end{align*}
i.e. $I_f (\mu^*) = M$ and $\mu^*$ is also a minimizer of $I_f$. 

Since $\mu^* \in \ext (K)$,  we can apply  Karr's theorem (Theorem \ref{thm:karr}) to finish the proof of the theorem. \end{proof}

\section{Empty interior of $p$-frame energy minimizers: the proof of Theorem \ref{t:pframe}}\label{sec:Empty interior, p-frame}

Conjecture \ref{conj1}, stating that the minimizers of the $p$-frame energy with $p\not\in 2 \mathbb N$ are necessarily discrete, remains open, outside of some specific cases covered in our companion paper \cite{BGMPV}. In the present section, we prove a weaker statement, namely,  that the support of every  minimizer of such energies has empty interior, i.e. Theorem \ref{t:pframe}.

A similar result has been proved in \cite{finsterSupport2013} for the energy on $\S^2$ with the kernel given by \eqref{eq:causal}. While our approach is inspired by theirs and the main line of reasoning  follows an analogous path, specific constructions and arguments  in the proofs of  Propositions \ref{prop:I} and \ref{prop:II} below are much more peculiar and significantly more involved in the case of the $p$-frame energy. 

We shall need  a standard fact from potential theory (\cite{Bjorck1956, borodachovMinimal, landkofFoun1972}). 
For a measure $\mu \in \mathcal P(\S^{d-1})$, let us define the {\em{potential}} $F_\mu$ of $\mu$ with respect to $f$ as
\begin{equation}\label{e.4.6}
F_\mu(x) : = \int_{\mathbb{S}^{d-1}} f( \langle x, y \rangle) d \mu(y),\,\,\, x\in \S^{d-1}.
\end{equation}
Notice that this meaning  of the term ``potential'' is consistent with our previous usage, since the function $f( \langle x, y \rangle)$ is just the potential generated by a unit point charge at $y$, i.e. $ f( \langle x, y \rangle) = F_{\delta_y} (x)$. It is a well-known phenomenon from electrostatics that the potential of the equilibrium measure is constant on the support of the measure.
\begin{lemma}\label{lem:const_pot} If $f \in C([-1,1])$ and $\mu$ is a minimizer of $I_f$, then the potential $F_{\mu}$ 
is constant on the support of $\mu$:
\begin{equation}
 F_{\mu} \big|_{\operatorname{supp} \mu} = \inf_{ x \in \mathbb{S}^{d-1}} F_{\mu}(x) = I_{f}(\mu).
 \end{equation}
\end{lemma}

  In what follows, the value of $p \in \mathbb R_+ \setminus 2 \mathbb N$ is fixed, $f(t) = |t|^p$, and $\mu$ is assumed to be a minimizer of  $I_f$.  The proof of Theorem \ref{t:pframe} is based on two properties of interior points of $\operatorname{supp} (\mu) $.

\begin{proposition}\label{prop:I}
Let $p \in \mathbb R_+ \setminus 2 \mathbb N$, $f(t) = |t|^p$, and $\mu$ be a minimizer of  $I_f$. Then for $z \in \big( \operatorname{supp} \mu  \big)^\circ$, $$\operatorname{supp} \mu  \cap  z^\perp  = \emptyset.$$
\end{proposition}

\begin{proposition}\label{prop:II}
Let the same  conditions as in Proposition~\ref{prop:I} hold. Then for $z \in \big( \operatorname{supp} \mu \big)^\circ$, $$\operatorname{supp} \mu  \cap  z^\perp  \neq \emptyset.$$
\end{proposition}

Since these two statements are  clearly mutually exclusive whenever $\operatorname{supp} \mu$ is non-empty, their validity proves Theorem \ref{t:pframe}, i.e. that there are no interior points in the support of a minimizer. The remainder of this section is dedicated to the proof of these propositions. 

We now sketch the argument for the first proposition. In short, the idea of the proof is the following. Assume that there exists a point $y \in \operatorname{supp} \mu$ such that  $\langle y ,  z \rangle= 0$. We shall construct a finite set of points $X = \big\{x_i \big\}_{i=1}^N 
 \subset \operatorname{supp} \mu$, 
 such that the matrix $\big[ f (\langle x_i , x_j \rangle ) \big]_{i,j}$ 
 is not positive semidefinite, thus violating Lemma \ref{lem:positive_semidef}.  The set $X$ will consist of the points $z$, $y$, and a number (depending on $p$) of points, equidistantly spaced around $z$ on the great circle connecting $y$ and $z$. We now make this precise.

\begin{proof}[The proof of Proposition \ref{prop:I}]  
We prove that if $z$ is an interior point of  a minimizer's support, then the orthogonal hyperplane $z^\perp$ does not intersect the support of $\mu$. 

Fix $z$ in the interior of $\operatorname{supp} (\mu)$ and let $y \in \mathbb S^{d-1}$  be any point such that $\langle y, z\rangle = 0$. Setting $k \in \mathbb{N}$ so that $2k-2 < p < 2k$, we shall construct a set $\{ x_0, \dots, x_{N-1} \}$ of $N = 2k + 2$ points, all of which lie on the great circle connecting $z$ and $y$. The points $x_0,\dots, x_{2k}$ are chosen in such a way that the  angle between $x_j$ and $z$ is $ (j-k)\varepsilon$ for some small $\varepsilon >0$. Thus $x_k = z$, and the points $x_0$ and $x_{2k}$ make angles $-k\varepsilon$ and $k\varepsilon$ with $z$, respectively. Observe that when $\varepsilon$ is small enough, all of these points $x_0,\dots,x_{2k}$ belong to  $\operatorname{supp} (\mu)$, since $z$ is an interior point. Finally, we set $ x_{2k+1} = y$. Then the angle between  $ x_{2k+1} = y$ and $x_j$, $j=0,\dots, 2k$, is $\frac{\pi}2 - (j-k) \varepsilon$. In order to apply Lemma \ref{lem:positive_semidef}, we consider the matrix $ \displaystyle{A  =  \big[ f ( \langle x_i , x_j \rangle)  \big]_{i,j=0}^{2k+1}}$.

We will show  that the matrix $A$ is not positive semidefinite. To this end, we first  construct an auxiliary vector  $v \in \mathbb{R}^{2k+1} \setminus \{0 \}$ such that for $m \in \{ 0, 1, ..., 2k-1 \}$,
\begin{equation}\label{e.4.10}
 \sum_{j=0}^{2k} j^m v_j = 0,
 \end{equation}
i.e.\  this vector must be in the (right) kernel of the Vandermonde matrix
$$ \begin{pmatrix}
1 & 1 & 1 & 1 & \cdots & 1 \\
0 & 1 & 2 & 3 &\cdots & 2k \\
0 & 1 & 2^2 & 3^2 & \cdots & (2k)^2 \\
\vdots & \vdots & \vdots & \vdots & \ddots & \vdots \\
0 & 1 & 2^{2k-1} & 3^{2k-1} & \cdots & (2k)^{2k-1}
\end{pmatrix}.$$
We can take  the entries of $v$  to be 
\begin{equation}\label{e.4.11} 
 v_j = \prod_{\substack{i=0 \\ i \neq j}}^{2k} \frac{1}{j-i} =  \frac{(-1)^j}{(2k-j)! j!}.
\end{equation}
Such a vector can be seen to be in the kernel of the matrix above by use of the formula for the inverse of the square Vandermonde matrix (see Ex. 40 on page 38 of \cite{knuth}).

Consider a vector $u = [ \alpha v_0, \alpha v_1, ..., \alpha v_{2k}, \beta]^T \in \mathbb{R}^{2k+2}$, where $\alpha$, $\beta \in \mathbb R$.  Then we have
\begin{equation}\label{e.4.12}
\begin{aligned}
\left\langle  Au,u  \right\rangle & = \alpha^2 \left( \sum_{i,j=0}^{2k} v_i v_j f( \langle x_{i}, x_j \rangle) \right) + 2 \alpha \beta \left( \sum_{j=0}^{2k} v_j f( \langle x_{2k+1}, x_j \rangle) \right) + \beta^2.
\end{aligned}
\end{equation}
We shall show that the real numbers $\alpha$ and $\beta$ can be chosen in such a way that the expression above is negative, for $\varepsilon$ sufficiently small. 

Observe that for $ i$, $j =  0, \dots, 2k$ we have $$  f( \langle x_i, x_j \rangle) = \cos^p \big( (i-j) \varepsilon \big). $$
Since $\cos^p(t)$ is even, smooth near zero, and $\cos^p(0) = 1$, we can use its Taylor expansion to estimate the first term of \eqref{e.4.12} as follows 
\begin{equation}\label{e.4.13}
\begin{aligned}
\sum_{i,j=0}^{2k} v_i v_j f( \langle x_{i}, x_j \rangle) & = \sum_{i,j=0}^{2k} v_i v_j \cos^p \big( (i - j) \varepsilon \big) \\
& = \sum_{i,j=0}^{2k} v_i v_j \left( 1 + \sum_{m=1}^{2k-1} a_m \varepsilon^{2m} (i-j)^{2m} + O(\varepsilon^{4k})  \right) \\
& = \left( \sum_{j=0}^{2k}  v_j \right)\left( \sum_{i = 0}^{2k} v_i \right)  + \sum_{m=1}^{2k-1} a_m \varepsilon^{2m} \left( \sum_{i,j=0}^{2k} v_i v_j (i-j)^{2m} \right)  +  O(\varepsilon^{4k})  \\
& = \sum_{m=1}^{2k-1} a_m \varepsilon^{2m} \sum_{i,j=0}^{2k} v_i v_j \sum_{l=0}^{2m} \binom{2m}{l} i^l j^{2m-l}  +  O(\varepsilon^{4k})  \\
& = \sum_{m=1}^{2k-1} a_m \varepsilon^{2m} \sum_{l=0}^{2m} \binom{2m}{l} \left( \sum_{i=0}^{2k} v_i   i^l \right) \left( \sum_{j=0}^{2k} v_j  j^{2m-l} \right)  +  O(\varepsilon^{4k})  \\
& =  O(\varepsilon^{4k}),  \\
\end{aligned}
\end{equation}
where we have used the fact that   for all  values of $l =0, 1, \dots, 2m$, either $l \leq 2k-1$ or $2m-l \leq 2k-1$.\\

We now turn to the second term of \eqref{e.4.12}. Observe that for $j=0, \dots, 2k$ we have
\begin{equation}\label{e.4.14}
 f( \langle x_{2k+1}, x_j \rangle) = f( \langle y , x_j \rangle)  =  \cos^p \bigg(\frac{\pi}2 - (j-k) \varepsilon \bigg)      =  \big| \sin \big( (j-k) \varepsilon \big)     \big|^p.
\end{equation}
We then find that 
\begin{equation}\label{e.4.15}
\begin{aligned}
\sum_{j=0}^{2k} v_j f( \langle y, x_j \rangle) & = \sum_{j=0}^{2k} v_j \sin^p \left(  |k-j| \varepsilon \right)   = \sum_{j=0}^{2k} v_j \left(  { |k-j| \varepsilon} + O( \varepsilon^3) \right)^p \\
& = \sum_{j=0}^{2k} v_j \left(   |k-j| \varepsilon \right)^p \left(1+ O( \varepsilon^2) \right)^p  =   \varepsilon^p \sum_{j=0}^{2k}  v_j |k-j|^p + O( \varepsilon^{p + 2}) .
\end{aligned}
\end{equation}
We now analyze the coefficient  of $\varepsilon^p$ in  the above expression using \eqref{e.4.11} 
\begin{equation*}
\begin{aligned}
\sum_{j=0}^{2k}  v_j |k-j|^p & = 
  \sum_{j=0}^{2k} (-1)^j \frac{|k-j|^p}{(2k-j)! j!}   = 2 \sum_{j=0}^{k-1} (-1)^j \frac{(k-j)^p}{(2k-j)! j!}. \\
\end{aligned}
\end{equation*}
\vskip2mm

\noindent Since the above is a sum of $k$ exponential functions of $p$, we know that  $ \sum_{j=0}^{k-1} (-1)^j \frac{(k-j)^p}{(2k-j)! j!}$ has at most $k-1$ zeros, see e.g. Ex. 75 from \cite[pg. 46]{polyaszego}.  We will show that these zeros are exhausted by the even integer values $p=2, 4, \dots, 2k-2$.  Indeed, assume  indirectly  that
$$ 2 \sum_{j=0}^{k-1} (-1)^j \frac{(k-j)^p}{(2k-j)! j!} = b \neq 0$$
for some even integer $0 < p \leq 2k-2$. Then according to \eqref{e.4.12}, \eqref{e.4.13}, and \eqref{e.4.15} we have
\begin{equation*}
\begin{aligned}
\left\langle  Au, u  \right\rangle & = \alpha^2 O(\varepsilon^{4k}) + 2 \alpha \beta \left( b \varepsilon^p + O( \varepsilon^{p + 2}) \right) + \beta^2.
\end{aligned}
\end{equation*}
Since $p<2k$, for $\varepsilon$ sufficiently small, the discriminant of this quadratic form is positive, hence we can choose $\alpha$ and $\beta$ so that $\langle Au, u \rangle < 0$.  However, since $f(t) = |t|^p$ is a positive definite function on $\S^{d-1}$ for even integer  $p$, this is a contradiction, as the matrix $A$ must be positive semidefinite for any collection $\{ x_i\}$. Therefore $$ 2 \sum_{j=0}^{k-1} (-1)^j \frac{(k-j)^p}{(2k-j)! j!} = 0$$
for all $p \in \{2, 4, ..., 2k-2 \}$. Since there are at most $k-1$ zeros of this function, we then know that
$$  b_p :=  2 \sum_{j=0}^{k-1} (-1)^j \frac{(k-j)^p}{(2k-j)! j!} \neq 0$$
for all other  values of $p$. Let $p \in (0, 2k) \setminus \{2, 4, ..., 2k-2\}$. Then 

\begin{equation*}
\begin{aligned}
\left\langle   Au, u  \right\rangle & = \alpha^2 O(\varepsilon^{4k}) + 2 \alpha \beta \left( b_p \varepsilon^p + O( \varepsilon^{p + 2}) \right) + \beta^2,
\end{aligned}
\end{equation*}
and by the previous argument, for $\varepsilon$ sufficiently small, we could choose $\alpha$ and $\beta$ so that $\langle Au, u \rangle < 0$, i.e. $A$ is not positive definite. Thus, according to Lemma~\ref{lem:positive_semidef},  $\{ x_0, x_1, ..., x_{2k}, y \}$ is not a subset of $ \operatorname{supp} \mu$. Since, by assumption, for small $\varepsilon >0$  the points $x_0, x_1, \dots , x_{2k}$ all lie in a neighborhood of $z$ and hence in $\operatorname{supp} \mu$, this implies that $y \not\in \operatorname{supp} \mu $ and so  $\operatorname{supp} \mu  \cap  z^\perp  = \emptyset$. \end{proof}

\noindent  {{We would like to make the following remark.}} 
Observe that for   $p\not\in 2 \N$, the number of points used to disprove positive definiteness of $f(t) = |t|^p$ in the  argument above is of the order $p$. A restriction of this type is actually  necessary.  Indeed, according to the result of Fitzgerald and Horn \cite{fitzgeraldFract1977}, for any positive definite matrix $A = \big[ a_{ij} \big]_{i,j=1}^N$ with non-negative entries $a_{ij} \ge 0$, its Hadamard powers $ A^{(\alpha)} =  \big[ a_{ij}^\alpha \big]_{i,j=1}^N$ are also positive definite when $\alpha \ge N-2 $. Let $G =  \big[ \langle x_i , x_j \rangle \big]_{i,j=1}^N$ be the Gram matrix of the set $X = \{ x_1,\dots, x_N \} \subset \mathbb S^{d-1}$. Since the matrix $G^{(2)} =  \big[ |\langle x_i , x_j \rangle|^2  \big]_{i,j=1}^N$ is positive definite and has non-negative entries, we have that the matrix $$ G^{(p)}  =  \big[ |\langle x_i , x_j \rangle|^p  \big]_{i,j=1}^N = \big( G^{(2)} \big)^{(p/2)} $$ is positive definite whenever $p/2 \ge N-2$. Therefore, to obtain a non-positive definite matrix  $G^{(p)}$, we must take $N \ge 2+ p/2$ points.

We now complete the proof of Theorem \ref{t:pframe}.

\begin{proof}[The proof of Proposition \ref{prop:II}] 
Suppose a neighborhood of a point $z \in \mathbb{S}^{d-1}$ is contained in the support of $\mu$. We shall demonstrate that $\operatorname{supp}  \mu$ must intersect the hyperplane $z^\perp$. 

Let us assume the contrary, i.e. $\operatorname{supp} \mu \cap z^\perp = \emptyset$. We may move all the mass of $\mu$ to the hemisphere centered at $z$ by defining a new measure  $\mu_z \in \mathcal{P} (\S^{d-1})$:
$$ \mu_z( E) = \begin{cases} 
\mu( -E \cup E), & \textup{if } E \subseteq \{ x\in \mathbb S^{d-1} :\, \langle z, x\rangle >0 \},  \\
\mu(E), & \textup{if } E \subseteq z^\perp,  \\
0,  & \textup{if } E \subseteq  \{ x\in \mathbb S^{d-1} :\, \langle z, x\rangle < 0 \}.
\end{cases}$$
Since $f(\langle z, y \rangle) = f( \langle z, -y \rangle)$ for all $y \in \mathbb{S}^{d-1}$, this does not change the energy, i.e.  $ I_{f}( \mu_z) = I_{f}(\mu)$, so that $\mu_z$ is also  a minimizer.\\

Since  $\displaystyle{\operatorname{supp} \mu \cap z^\perp = \emptyset}$,   we also have that $\displaystyle{\operatorname{supp} \mu_z   \cap z^\perp = \emptyset}$, i.e.  $\operatorname{supp} \mu_z   \subset \{ x\in \mathbb S^{d-1} :\, \langle z, x\rangle >0 \}$. Compactness of   the support of $\mu_{z}$ then implies that it is separated from $z^\perp$, i.e.\ for some $\delta >0$  we have  $\langle y , z \rangle > \delta$ for each $y \in \operatorname{supp} \mu_z  $. Let us choose an open neighborhood $U_z$ of $z$, small enough so that $U_z \subset \operatorname{supp} \mu_z  $ and so that for each $x\in U_z $ and each $y \in \operatorname{supp} \mu_z$, $\langle y , x \rangle > \delta >0$.

We can now write the potential \eqref{e.4.6} of $\mu_z$ at the point $x \in U_z$ as 
\begin{equation}
 F_{\mu_z}(x) = \int_{ \mathbb{S}^{d-1} } | \langle x, y \rangle |^p \,  d \mu_z(y) = \int_{ \operatorname{supp}  \mu_z  }  \langle x, y \rangle^p  \, d \mu_z(y).
 \end{equation}
The discussion above implies that the last expression is well-defined for all $p>0$.  According to Lemma~\ref{lem:const_pot}, the potential   $F_{\mu_z} (x) $ is constant on $U_z \subset  \operatorname{supp} \mu_z $.\\

When $p$ is an odd integer, the proof can be finished very quickly. Indeed, in this case the expression  $$ g(x) = \int_{ \operatorname{supp} \mu_z  }  \langle x, y \rangle^p  \, d \mu_z(y)$$  is well defined for each $x\in \mathbb S^{d-1}$ and yields an  analytic function on the sphere (actually, a polynomial). Hence, being constant on an open set, implies that it is is constant on all of $\mathbb S^{d-1}$, which is not possible since, obviously, $ g(-z) = - g(z) = - F_{\mu_z} (z) = - I_f (\mu_z) \neq 0$.  Compare this argument to Theorem \ref{Analytic_supp}.   \\

We now will present an approach which works for all $p\in \mathbb R_+ \setminus 2 \mathbb N$. Assume that there exists a differential operator $D$ acting on functions on the sphere with the following two properties:
\begin{enumerate}[(i)]
\item\label{da} $D$ locally annihilates constants, i.e.\ if $u(x)$ is constant on some open set $\Omega$, then $D_x u = 0$ on $\Omega$;
\item\label{db} $D_x  \left(  \langle x, y \rangle^p \right) < 0$  for all  $x\in U_z $ and  $y \in \operatorname{supp} \mu_z  $. 
\end{enumerate}

Existence of such an operator would finish the proof since we would then have for each $x\in U_z$
\begin{equation}\label{e.4.17}
0 = D_x F_{\mu_z } (x) =  \int_{ \operatorname{supp}(\mu_z )} D_x  \left(  \langle x, y \rangle^p  \right) \, d \mu_z(y) <0,
\end{equation}
which is a contradiction. Note that that switching to $D_x  \left(  \langle x, y \rangle^p \right) > 0$ in condition (\ref{db}) does not affect the proof. \\

We now construct such an operator $D$. Let $\Delta$ denote the Laplace--Beltrami operator on $\mathbb S^{d-1}$. Writing it in the standard spherical coordinates $\theta_1,\dots,\theta_{d}$ one obtains (see, e.g., equation (2.2.4) in \cite{KMR})
\begin{equation}\label{e.4.18}
 \Delta = \sum_{j=1}^{d-1} \frac{1}{q_j (\sin \theta_j )^{d-1-j} } \frac{\partial}{\partial \theta_j}  \left(  (\sin \theta_j )^{d-1-j}  \frac{\partial}{\partial \theta_j}  \right),
 \end{equation}
  where $q_1 =1 $ and $q_j = (\sin \theta_1 \dots \sin \theta_{j-1})^2$ for $j>1$.

For a fixed $y\in \mathbb S^{d-1}$, choose the coordinates so that $\cos \theta_1 = \langle y, x\rangle$. Then $ \langle y, x\rangle^p = \cos^p \theta_1$, effectively leaving just one term in the formula above, and a direct computation shows that 
 \begin{equation}\label{e.4.19}
\Delta_x   \left(  \langle x, y \rangle^p \right)  = p (p-1) \langle x, y \rangle^{p-2} - p (p+d-1) \langle x, y \rangle^p.
\end{equation}
Observe that if $p\in (0,1]$, then the operator $\Delta_x$ satisfies conditions (\ref{da}) and (\ref{db}), hence completing the proof for this range of $p$.\\

Now consider the operator 
$D = \Delta \, \big( \Delta +  p (p+d -1) \big)$.
It is easy to see that 
\begin{equation}\label{e.4.20}
\begin{aligned}
    \Delta_x \, \big( \Delta_x +  p (p+d -1) \big)   \left(  \langle x, y \rangle^p \right)  & = p (p-1)  \Delta_x   \left(  \langle x, y \rangle^{p-2} \right)  \\ 
 & = p(p-1) (p-2) \langle x, y \rangle^{p-4} \cdot \big(  (p-3)  -  (p+d-3) \langle x, y \rangle^2\big). 
\end{aligned}
\end{equation} 

If $p\in (2, 3]$, then $p-3 \le 0$ and $p+d -3 > d-1 \ge 0$, so the expression above is strictly negative. Hence this operator satisfies conditions (\ref{da}) and (\ref{db}) for $2<p\le 3$. 

Moreover, if $p \in (1,2)$, the expression above is strictly positive. Indeed, the function $g_p(t) = (p-3) - (p+d-3) t $ is monotone on $[0,1]$ with $g_p (0) = p-3 < 0$ and $g_p (1) = - d < 0$. Therefore, condition (\ref{db}) holds with an opposite inequality sign, so the case $1<p<2$ is also covered.\\

It is now clear how to iterate this process. 
  Define now the operator $D^{(0)} = \Delta$, $D^{(1)} = \Delta \, \big( \Delta +  p (p+d -1) \big)$, and, more generally, for $k\in \mathbb N$, define the differential operator of order $2k+2$ 
\begin{equation}\label{e.4.21}
D^{(k)}  = \Delta  \left( \Delta + (p+d-2k+1) \prod_{j=0}^{2k-2} (p-j) \right) \cdots \Bigg( \Delta + p (p-1) (p-2) ( p+d-3) \Bigg) \big( \Delta  + p ( p+d-1) \big).
\end{equation} 
Let $p \in \mathbb R_+\setminus 2 \mathbb N$ and choose $k \in \mathbb N_0$ so that $2k-1 <p \le  2k+1$. An iterative computation  shows that 
\begin{equation}\label{e.4.22}
\begin{aligned}
 D^{(k)}_x   \left(  \langle x, y \rangle^p \right) & = \left( \prod_{j=0}^{2k+1} (p-j)\right) \,  \langle x, y \rangle^{p-2k-2} - \left( \prod_{j=0}^{2k} (p-j)\right) (p + d -  2k -1)\,  \langle x, y \rangle^{p-2k} \\ 
 & = \left( \prod_{j=0}^{2k} (p-j)\right) \,  \langle x, y \rangle^{p-2k-2} \cdot \left(  (p- 2k-1)  - (p + d -  2k -1 )\,  \langle x, y \rangle^{2} \right) .
 \end{aligned}
\end{equation} 

For $p\in (2k, 2k+1]$, the expression above is strictly negative, since  $p- 2k -1 \le 0 $ and $p+ d- 2k -1 > d-1 \ge 0$. 

At the same time, for $p \in (2k-1, 2k)$, this expression is strictly positive, because  $\prod_{j=0}^{2k} (p-j) < 0 $ and the monotone function $g_p (t) =  (p- 2k-1)  - (p + d -  2k -1 )t $ takes values $g_p (0) = p - 2k-1 <0$ and $g_p (1) = - d <0$. Thus, operator $D^{(k)}$ allows us to prove Propostion \ref{prop:II} for $p$ in the range $ (2k-1,2k) \cup (2k, 2k+1]$.
\end{proof}

We suspect that an analog of Theorem \ref{t:pframe} also holds for the Fejes T\'oth conjecture \cite{FT} mentioned in the introduction. Recall that this conjecture (its continuous version) deals with the energy $I_f$ with potential $f(t) = \arcsin |t|$ and speculates that the discrete measure uniformly concentrated on the elements of an orthonormal basis minimizes $I_f$. If   the conjecture is true, not all the  minimizers of this energy are discrete. For example, as observed  in \cite{BM}, on $\S^3$,  normalized uniform $1$-dimensional Hausdorff measure on two orthogonal copies of $\S^1$, i.e. on the set $$ \{ (x_1, x_2, 0,0 ):\, x_1^2 +x_2^2 =1 \} \cup  \{ (0,0,x_3, x_4  ):\, x_3^2 +x_4^2 =1 \}, $$ would also yield a minimizer. This effect is related to the fact that, while the kernel $f(t) = \arcsin |t|$ is  {\em{not}} positive definite on $\S^{d-1}$ with $d\ge 3$, it is indeed positive definite on $\S^1$, i.e. the uniform measure is a minimizer on the circle. Thus, assuming the conjecture, this energy does have non-discrete minimizers.

\section{Minimizers of  energies with analytic kernels}\label{sec:analytic}
We can also prove a statement analogous to Theorem \ref{t:pframe} for  a wide class of energies -- namely, those with analytic potentials. 

\begin{theorem}\label{Analytic_supp}
Assume that  $f$ is a real-analytic  function on $[-1,1]$, such that $\sigma $ is not a minimizer of $I_f$, i.e. $f$ is not (up to an additive constant) positive definite on $\mathbb{S}^{d-1}$. Let  $\mu$ be a minimizer of $I_{f}$, then $(\operatorname{supp}(\mu))^{\circ} = \emptyset$. Moreover, when $d =2$, then $\operatorname{supp}(\mu)$ must be discrete
\end{theorem}

\begin{proof}

Suppose, indirectly, that $(\operatorname{supp}(\mu))^{\circ} \neq \emptyset$. By Lemma \ref{lem:const_pot}, we know that the potential
$$F_{\mu}(x) = \int_{ \mathbb{S}^{d-1}} f( \langle x , y \rangle) d \mu(y)$$
is constant on $\operatorname{supp}(\mu)$. Since $f(\langle x, y \rangle)$ is real-analytic on $\mathbb{S}^{d-1} \times \mathbb{S}^{d-1}$, $F_{\mu}(x)$ is real-analytic on $\mathbb{S}^{d-1}$. Since $F_\mu$ is real-analytic and constant on an open set in $\mathbb{S}^{d-1}$, it is constant on all of $\mathbb{S}^{d-1}$ \cite[Lemma 2.4]{quinto97}. In addition, $F_\sigma (x) = I_f (\sigma)$ is constant on $\S^{d-1}$ due to rotational invariance. We then obtain
\begin{align*}
I_f(\mu) & = \int_{ \mathbb{S}^{d-1}} \int_{\mathbb{S}^{d-1}} f( \langle x, y \rangle ) d \mu(y) d\mu(x) = \int_{ \mathbb{S}^{d-1}} F_{\mu}(x) d\mu(x)  = \int_{ \mathbb{S}^{d-1}} F_{\mu}(x) d\sigma(x) \\
& = \int_{ \mathbb{S}^{d-1}} \int_{\mathbb{S}^{d-1}} f( \langle x, y \rangle ) d \mu(y) d\sigma(x) = \int_{ \mathbb{S}^{d-1}} \int_{\mathbb{S}^{d-1}} f( \langle x, y \rangle ) d\sigma(x) d \mu(y) \\ 
&  = \int_{ \mathbb{S}^{d-1}} F_\sigma (x)  d \mu(y)   = I_f( \sigma).
\end{align*}
This is clearly a contradiction, as by the assumption, 
$I_f $ is not minimized by $\sigma$. Our first claim then follows.

For $\mathbb{S}^1$, we have that if $F_\mu$ is constant on a set $\{ z_1, z_2, ... \} \subset \mathbb{S}^1$ with an accumulation point, $F_\mu$ is constant on $\mathbb{S}^1$. The proof of our second claim then follows as above.
\end{proof}

If $\S^{d-1}$ is replaced with one of the projective spaces $\FP^{d-1}$ ($\F = \R$ or  $\C$, for instance) a similar result can be derived as above. In this case kernels $f$ are also functions of the cosine of the geodesic distance $\tau(x,y)=2|\langle x,y\rangle|^2-1$ under identification of points with unit vectors $x,y\in\F^{d}$; see \cite{BGMPV} for more details on energy integrals over these spaces. \\

In the spirit of Theorem \ref{Analytic_supp}, as well as  Corollary \ref{cor:dich}, it may be tempting to conjecture that if $f$ (not necessarily analytic) is not positive definite on $\S^{d-1}$ (up to constant), i.e $I_f(\mu)$ is {\em{not}} minimized by $\sigma$, then the support of any minimizer of $I_f$ must have empty interior. However, this is not true, as the following simple example shows. Assume that $f\in C[-1,1]$  is constant near $t=1$ and strictly decreasing otherwise, i.e. it satisfies for some fixed $\gamma  \in (0,1)$, $$ f(1) = f(\tau) = \min_{t\in[-1,1]} f(t) \,\, \textup{ for any }\,\, \tau \in [1- \gamma,1],$$ and $f(\tau) > f(1)$ for all $\tau \in [-1, 1- \gamma)$. It is then evident that for any $z \in\S^{d-1}$  $$ \min_{\mu \in \mathcal P (\S^{d-1}) } I_f (\mu ) = I_f \big(\delta_z \big) = f(1),$$
and  
$I_f (\sigma) > I_f (\delta_x)$, i.e. $\sigma $ is not a minimizer of $I_f$. Let $C(z, \alpha) = \{ x\in \S^{d-1}:\, \langle x,z \rangle >\alpha \} $ denote the {\em{spherical cap}} of ``height'' $\alpha $ centered at $z\in S^{d-1}$. Let $\nu$ be the  normalized uniform measure on $C(z,\alpha)$, i.e.  $$d\nu (x) = \frac{{\bf{1}}_{C(z,\alpha)} (x)}{\sigma \big( C(z,\alpha) \big)} d\sigma (x), $$ with $\alpha = 1 - \frac{\gamma}{4}$. Then for each $x,y \in C(z,\alpha)$, we have $\langle x,y \rangle  > 1-\gamma$, and hence $$ I_f (\nu ) = I_f \big(\delta_z\big) = f(1),$$ i.e. $\nu$ is also a minimizer of $I_f$, but its support has non-empty interior.\\

\section{Applications of the results to energies with polynomial kernels}\label{sec:poly}

We observe that the results of Sections \ref{sec:discrete} and \ref{sec:analytic} apply  if $f$ is a polynomial.  Indeed, Theorem \ref{Analytic_supp} is applicable since polynomials are analytic, while  the conditions of Theorem \ref{thm:discrete} hold because the Gegenbauer expansion  has only finitely many terms. We summarize these statements in the following corollary. 

\begin{corollary}\label{cor:poly}
Assume that $f$ is a polynomial whose Gegenbauer expansion is 
\begin{equation}
f(t) = \sum_{n=0}^m a_n C_n^\lambda (t). 
\end{equation}
\begin{enumerate}[(i)]
\item\label{poly1} There exists a discrete minimizer $\mu \in \mathcal P (\S^{d-1})$ with  
\begin{equation*}\label{eq:wd}
 \big| \supp \mu \big| \le 1 + \sum_{\substack{\{n:\ a_n >0,\ 1\leq n\leq m \} }}  \dim \mathcal H_n^d.  
 \end{equation*}
\item\label{poly2} If, moreover, $\sigma$ is not a minimizer of $I_f$ over $\mathcal P (\S^{d-1})$, i.e. there exists $n\ge 1 $ such that $a_n <0$, then the support of any minimizer of $I_f$ has empty interior.  For $\S^1$ the support is finite.
\end{enumerate}
\end{corollary}

We observe that when $a_n > 0$  for $n=1,\dots,m$, i.e. $f$ is positive definite on $\S^{d-1}$  polynomial (up to constant), the statement of Theorem \ref{thm:discrete} (and hence also part \eqref{poly1} of the above corollary) is well known. In this case, the discrete minimizers $\mu = \sum \omega_{x_i} \delta_{x_i}$ are exactly {\em{weighted spherical $m$-designs}}, i.e.\ for any polynomial $P$ of degree at most $m$ we have  $$ \sum_{i} \omega_{x_i} P(x_i)  = \int_{\S^{d-1}} P(x)  d\sigma (x). $$

A certain well-known generalization of this fact can also be easily deduced from part \eqref{poly1} of Corollary \ref{cor:poly}. Let $M \subset \mathbb N_0$ with $0\in M$.  Call a set $\{x_i\}_{i=1}^k\subset \S^{d-1}$ with positive weights $\omega_{x_i}$ a {\it{weighted $\mathcal{M}$-design}} if for every $m\in \mathcal{M}$ and for every spherical harmonic $Y \in \mathcal H_m^d$ one has $$ \sum_{i} \omega_{x_i} Y(x_i)  = \int_{\S^{d-1}} Y(x)  d\sigma (x). $$ When $\mathcal{M}=\{0,1,\dots,m\}$, this definition coincides with the definition of an $m$-design. Such objects arise naturally for some configurations. For example, the $600$-cell, one of the six $4$-dimensional convex regular polytopes with vertices which form a $120$-point subset of $\S^3$, yields an exact cubature formula for spherical harmonics of degrees up to $19$, {\it{excluding}} degree $12$. In other words, it is an $\mathcal{M}$-design for $\mathcal{M} = \{ 0,1,\dots, 11\} \cup\{13, \dots, 19\}$. By taking $a_n >0$ only for $n\in \mathcal{M}$ and applying part \eqref{poly1} of Corollary \ref{cor:poly}, one easily concludes existence of weighted $\mathcal{M}$-designs on the sphere $\mathbb S^{d-1}$ of cardinality at most
$\sum_{n\in \mathcal{M}} \dim \mathcal H_n^d$. This statement is encompassed by more general results \cite{tchaka,rogo}

Theorem \ref{thm:discrete} and part \eqref{poly1} of Corollary \ref{cor:poly} vastly generalize these well-known statements, essentially showing that the addition of  any number of negative definite terms does not destroy the statement: discrete minimizers with the same cardinality still exist. 

Concerning part \eqref{poly2} of Corollary \ref{cor:poly}, it might be interesting to give some explicit examples of polynomials $f$ with at least one negative coefficient $a_n<0$ for $n\ge 1$, for which the minimizers of $I_f$ are not necessarily discrete. Finally, we mention that the case of energy optimization for polynomial potentials in $d=2$ is more approachable than in higher dimensions, due to the classical solution of the trigonometric moment problem \cite[Theorem 1.4]{shohattam}.

\section{Local minimizers of the $p$-frame energy with $p\in 2\N$ are global.}\label{sec:locglob}

Finally, we make an observation that for energies with positive definite kernels, including the $p$-frame energy with $p\in 2 \N$, every local minimizer is necessarily global. We consider local minima in a rather general sense. 
\begin{definition}
We say that a probability measure  $\xi \in \mathcal P (\S^{d-1})$ is a local minimizer of $I_f$ if for each $\mu \in \mathcal P ( \S^{d-1})$ and for any $\tau>0$ small enough (depending on $\mu$), $$ I_f \big( \xi \big) \le  I_f \big( (1-\tau) \xi + \tau \mu  \big). $$
\end{definition}

Observe that this definition is satisfied if $\xi$ is a local minimum with respect to many reasonable metrics on $\mathcal P (\S^{d-1})$, i.e. if there exists $\varepsilon >0$ such that $I_f (\xi ) \le I_f (\mu)$ whenever $ d( \xi, \mu ) <\varepsilon$, where $d(\xi, \mu)$,  represents, for example,  the $d_p$-Wasserstein distance, $p<\infty$, or the total variation distance between measures.  The following proposition provides a relation between the local and global minimizers. 

\begin{proposition}\label{p:locglob}
Let $ f\in C[-1,1]$ and let $\nu \in \mathcal P (\S^{d-1})$ be a global minimizer of $I_f$. Assume also that $\xi \in \mathcal P (\S^{d-1})$ is a local minimizer of $I_f$ and that $\supp \xi \subset \supp \nu$. Then $\xi$ is also a global minimizer of $I_f$ over $\mathcal P (\S^{d-1})$.
\end{proposition}

If the function $f$ is positive definite (modulo a constant term) on the sphere $\S^{d-1}$, then the uniform measure $\sigma$ minimizes $I_f$ according to part  \eqref{part:energy} of Proposition \ref{p:pd}, hence one can take $\nu = \sigma $ in the lemma above.  Since $\sigma$ is supported on the whole sphere, this immediately leads to non-existence of local minimizers  which are not global:

\begin{corollary}\label{cor:locglob}
Let $ f\in C[-1,1]$ be positive definite on $\S^{d-1}$ (up to an additive constant)  and let $\xi$ be a local minimizer of $I_f$. Then $\xi$ is necessarily a global minimizer of $I_f$, i.e. $$ I_f (\xi ) = \min_{\mu \in \mathcal P (\S^{d-1})} I_f (\mu). $$
\end{corollary}

\begin{proof}[Proof of Proposition \ref{p:locglob}] Let $\nu$ be a global minimizer, that is $$I_{f}(\nu)=\inf\limits_{\mu\in\mathcal{P}(\S^{d-1})} I_{f}(\mu) = \alpha.$$ 
According to Lemma \ref{lem:const_pot}, the potential of $\nu$ satisfies
\begin{equation}\label{eq:pot_glob}
  F_\nu(x)  = \int_{\mathbb{S}^{d-1}} f( \langle x, y \rangle) d \nu(y)  =  I_f (\nu) = \alpha\,\,\, \quad \textup{for all $x \in\supp \nu$.}
  \end{equation}
Suppose, by contradiction, that $\xi$  satisfies $\alpha=I_{f}(\nu)<I_{f}(\xi)$. Since $\xi$ is a local minimizer, setting $\mu_\tau = \tau\nu+(1-\tau)\xi$, for sufficiently small $0< \tau<1$, we have
\begin{equation}\label{eq:minmulambda}
I_{f}(\mu_\tau )\geq I_{f}(\xi).
\end{equation}
Setting $I_{f}(\xi)=\beta >\alpha$ and  using  \eqref{eq:pot_glob}, 
a quick calculation shows that
\begin{align*}
I_f (\xi) & \le  I_f (\mu_\tau) =  \tau^2  I_f (\nu) + (1-\tau)^2 I_f (\xi) + 2\tau(1-\tau) \int_{\S^{d-1}}F_\nu (x) \, d\xi (x)  \\
&=\tau^2\alpha+(1-\tau)^2\beta +2\tau(1-\tau)\alpha.
\end{align*}
Thus, $\tau^2\alpha+2\tau(1-\tau)\alpha+(1-\tau)^2\beta\geq \beta$. 
However $$ \tau^2\alpha+2\tau(1-\tau)\alpha+(1-\tau)^2\beta < \beta \big(  \tau^2+2\tau(1-\tau)+(1-\tau)^2 \big) = \beta,$$ which is a contradiction. 
\end{proof}

Corollary \ref{cor:locglob} obviously applies to the $p$-frame energies  when $p= 2k$ is an even integer.    As discussed in the introduction,  $\sigma$  minimizes $I_f$, since $f(t) = t^{2k}$ is positive definite. Thus, all the local minimizers of the $2k$-frame energy are necessarily global. A somewhat  similar effect for $p=2$  has been observed  in \cite{benedetto2003finite} for discrete energies: it was proved that any finite configuration locally minimizing the $N$-point frame energy is also a global minimizer, and therefore it is a tight frame, whenever $N\geq d$.

\section{Acknowledgments}\label{sec:acknow}

We express our gratitude to the  following organizations that  hosted subsets of
the authors during the work on this paper: AIM, ICERM, INI, CIEM, Georgia Tech.
The first author was supported by the  grant DMS-1665007, the third author was
supported by the Graduate Fellowship 00039202,  and the fourth author was
supported in part by the grant DMS-1600693, all from the US National Science
Foundation. This paper is based upon work supported by the National Science Foundation
grant DMS-1439786 while D.B., A.G., and O.V.\ were in residence at the Institute
for Computational and Experimental Research in Mathematics in Providence, RI,
during the ``Point Configurations in Geometry, Physics and Computer Science''
program. This work is also supported in part by EPSRC grant no EP/K032208/1.  

We thank Henry Cohn, David de Laat, Axel Flinth, Michael Lacey, Svitlana Mayboroda, Bruno Poggi, Alexander Reznikov, Daniela Schiefeneder, Alexander Shapiro, and Yao Yao for helpful discussions.

\end{document}